\documentclass[10pt,twoside,leqno]{amsart}
\pdfoutput=1
\usepackage{amsfonts}
\usepackage{amsmath}
\usepackage{mathtools}
\usepackage{amsthm}
\usepackage{amssymb}
\usepackage{pifont}

\usepackage{mathrsfs}
\usepackage[numbers]{natbib}
\usepackage[fit]{truncate}
\usepackage{geometry}
\usepackage{hyperref}
\usepackage{bm}
\usepackage{tikz-cd}
\usepackage[makeroom]{cancel}
\usepackage{xcolor,colortbl}
\usepackage{hhline}
\usepackage{float}
\usepackage{enumitem} 
\usepackage{stackengine}

\usepackage{array, multirow}

\newcommand{\R}{\mathbb{R}}

\newcommand{\cm}[1]{\ignorespaces}
\newcommand{\lie}{\mathfrak{l}}

\theoremstyle{plain}
\newtheorem{theorem}{Theorem}[section]
\newtheorem{corollary}[theorem]{Corollary}
\newtheorem{lemma}[theorem]{Lemma}

\newtheorem{proposition}[theorem]{Proposition}

\theoremstyle{definition}

\newtheorem{example}[theorem]{Example}

\newtheorem{remark}[theorem]{Remark}
\numberwithin{equation}{section}

\begin{document}

\title[Locally conformally balanced metrics on almost abelian Lie algebras]{Locally conformally balanced metrics \\ on almost abelian Lie algebras} 
\author{Fabio Paradiso}
\address{Dipartimento di Matematica ``G. Peano''\\
	Universit\`a di Torino\\
	Via Carlo Alberto 10\\
	10123 Torino, Italy} \email{fabio.paradiso@unito.it}


\begin{abstract}
We study locally conformally balanced metrics on almost abelian Lie algebras, namely solvable Lie algebras admitting an abelian ideal of codimension one, providing characterizations in every dimension. Moreover, we classify six-dimensional almost abelian Lie algebras admitting locally conformally balanced metrics and study some compatibility results between different types of special Hermitian metrics on almost abelian Lie groups and their compact quotients. We end by classifying almost abelian Lie algebras admitting locally conformally hyperk\"ahler structures.
\end{abstract}

\maketitle
%\tableofcontents

\section{Introduction}
Let $(M,J)$ be a complex manifold of real dimension $2n$, $n \geq 2$, equipped with a Hermitian metric $g$ with associated fundamental $2$-form $\omega=g(J\cdot,\cdot)$. Its \emph{Lee form}, defined by $\theta=-d^*\omega \circ J$, is the unique $1$-form satisfying $d\omega^{n-1}=\theta \wedge \omega^{n-1}$.

A fundamental class of Hermitian metrics is provided by \emph{K\"ahler} metrics, satisfying $d\omega=0$. In literature, many generalizations of the K\"ahler condition have been introduced: two of them are the \emph{balanced} (or \emph{semi-K\"ahler}) condition, characterized by $d^*\omega=0$ (or equivalently $\theta=0$ or $d\omega^{n-1}=0$) and the \emph{locally conformally K\"ahler} (LCK) condition, namely $(M,J)$ admits an open cover $\{U_i\}$ and smooth maps $f_i \in C^{\infty}(U_i)$ such that $e^{-f_i}g\rvert_{U_i}$ is a K\"ahler metric on $(U_i,J\rvert_{U_i})$, where $g$ denotes the LCK metric.
The LCK condition is equivalently characterized by 
the conditions $d\omega=\frac{1}{n-1} \theta \wedge \omega$, $d\theta=0$. If $\theta$ is parallel with respect to the Levi-Civita connection, the LCK metric is called \emph{Vaisman}. For general results about LCK metrics, we refer the reader to \cite{DO, Orn, OV, AO1}.

A further weakening of both the balanced and the LCK conditions is given by the \emph{locally conformally balanced} (LCB) condition, whose definition is analogous to the one for the LCK condition and which is equivalently defined by $d\theta=0$. LCB metrics have been studied, for instance, in \cite{AU, AU1, FT1, LY, OOS, Oti, Shi, Shi1, Yang}. When $n=2$, balanced metrics are K\"ahler and LCB metrics are LCK. 

Recall also that a Hermitian metric is called \emph{strong K\"ahler with torsion} (SKT, also known as \emph{pluriclosed}) if $\partial \overline\partial \omega=0$ or, equivalently, if the torsion of the associated \emph{Bismut connection} vanishes. The Bismut connection $\nabla^B$ of a Hermitian manifold $(M,J,g)$ is the unique linear connection on $M$ having totally skew-symmetric torsion and satisfying $\nabla^Bg=0$, $\nabla^BJ=0$ (see \cite{Bismut,Gau}). Its associated \emph{Bismut-Ricci form} $\rho^B$ is the $2$-form locally defined by
\[
\rho^B(X,Y)=-\frac{1}{2} \sum_{i=1}^{2n} g(R^B(X,Y)f_i,Jf_i),\quad X,Y \in \Gamma_{\text{loc}}(TM),
\]
where $\{f_1,\ldots,f_{2n}\}$ is a local $g$-orthonormal frame and $R^{B}(X,Y)=[\nabla^B_X,\nabla^B_Y]-\nabla^B_{[X,Y]}$ denotes the curvature of $\nabla^B$.

A \emph{hypercomplex structure} on a smooth $4m$-dimensional manifold $M$ is given by a triple of (integrable) complex structures $(I_1,I_2,I_3)$ satisfying $I_1I_2I_3=-\text{Id}_{TM}$. A Riemannian metric on $M$ is called \emph{(locally conformally) hyperk\"ahler} (LCHK) if it is (locally conformally) K\"ahler with respect to the three complex structures and the three induced Lee forms coincide. Hypercomplex and hyperk\"ahler structures on Lie groups where studied for instance in \cite{Bar}, where four-dimensional Lie groups admitting left-invariant hypercomplex structures are classified, and \cite{BDF}, where, in particular, it is shown that left-invariant hyperk\"ahler metrics on Lie groups are flat.

We are interested in the case where $M$ is a simply connected almost abelian Lie group $G$ or a compact almost abelian \emph{solvmanifold}, namely a quotient $\Gamma \backslash G$, with $G$ a simply connected almost abelian Lie group and $\Gamma$ a lattice of $G$, i.e., a discrete subgroup of $G$. A connected (solvable) Lie group $G$ is called \emph{almost abelian} if it admits an abelian normal subgroup of codimension one, or equivalently if the Lie algebra $\mathfrak{g}$ of $G$ admits an abelian ideal $\mathfrak{n}$ of codimension one, so that $\mathfrak{g}$ is isomorphic to the semi-direct product $\R^k \rtimes_D \R$ for some $D \in \mathfrak{gl}_k$. If $\mathfrak{g}$ is non-nilpotent, such an ideal is unique and coincides with the nilradical of $\mathfrak{g}$.

A left-invariant Hermitian structure $(J,g)$ on $G$ or $\Gamma \backslash G$ descends to a structure on the Lie algebra $\mathfrak{g}$ of $G$, so that one can speak of Hermitian structures on $\mathfrak{g}$.
When $\mathfrak{g}$ is almost abelian of real dimension $2n$, as shown in \cite{LV}, these can be fully characterized in terms of the matrix associated with $\text{ad}_{e_{2n}}\rvert_{\mathfrak{n}}$ with respect to some fixed unitary basis $\left\{ e_1,\ldots, e_{2n} \right\}$ \emph{adapted} to the splitting $\mathfrak{g}= J\mathfrak{k} \oplus \mathfrak{n}_1 \oplus \mathfrak{k}$, where $\mathfrak{k} \coloneqq \mathfrak{n}^{\perp_g}$ and $\mathfrak{n}_1 \coloneqq \mathfrak{n} \cap J\mathfrak{n}$, and such that $Je_i = e_{2n+1-i}$, $i=1,\ldots,n$.

K\"ahler, SKT, balanced and LCK almost abelian Lie algebras were studied in terms of the data $(a,v,A)$ in \cite{LV,AL,FP,FP1,AO}. Six-dimensional almost abelian Lie algebras admitting SKT structures were classified in \cite{FP}, and in \cite{FS} the result was extended to a wider class of two-step solvable Lie algebras. For the classification of six-dimensional almost abelian Lie algebras carrying balanced structures, see \cite{FP1}.

In Section \ref{sec_lcb} we characterize LCB almost abelian Lie algebras in terms of the aforementioned algebraic data and in terms of the behaviour of the associated Bismut-Ricci form. 

In the following section, we classify six-dimensional almost abelian Lie algebras admitting LCK structures and those admitting LCB structures, building on the classification of six-dimensional almost abelian Lie algebras admitting complex structures in \cite{FP}, and remark which of the corresponding Lie groups admit compact quotients by lattices. 

In \cite{OOS}, the authors investigate the existence of two different types of special Hermitian metrics on a fixed compact complex nilmanifold (namely, the quotient of a simply connected nilpotent Lie group by a lattice): in Section \ref{sec_comp}, we consider analogous questions for almost abelian solvmanifolds, highlighting similarities and differences with respect to the nilpotent setting.

Finally, in Section \ref{LCHK} we study LCHK structures on almost abelian Lie algebras, giving a classification result in every dimension.

\smallskip
\emph{Acknowledgements}. The author would like to thank Anna Fino for suggesting the subject of this paper and for many useful comments and discussions. The author is also grateful to an anonymous referee for useful comments. The author was supported by GNSAGA of INdAM.

\section{Locally conformally  balanced metrics} \label{sec_lcb}

Let $\mathfrak{g}$ be a $2n$-dimensional almost abelian Lie algebra with a fixed abelian ideal $\mathfrak{n}$ of codimension one. Assume $(J,g)$ is a Hermitian structure on $\mathfrak{g}$ and denote by $\mathfrak{n}_1 \coloneqq \mathfrak{n} \cap J\mathfrak{n}$ the maximal $J$-invariant subspace of $\mathfrak{n}$, which does not depend on the metric $g$.
Then, as shown in \cite{LV}, with respect to a unitary basis $\{e_1,\ldots,e_{2n}\}$ for $\mathfrak{g}$ such that $\mathfrak{n}=\text{span}\left<e_1,\ldots,e_{2n-1}\right>$, $\mathfrak{n}_1=\text{span}\left<e_2,\ldots,e_{2n-1}\right>$, $Je_i=e_{2n+1-i}$, $i=1,\ldots,n$, the matrix $B$ associated with $\text{ad}_{e_{2n}}\rvert_{\mathfrak{n}}$ is of the form
\begin{equation} \label{B}
	B=\begin{pmatrix} a & 0 \\ v & A \end{pmatrix}, \quad a \in \R,\,v \in \mathfrak{n}_1,\, A \in \mathfrak{gl}(\mathfrak{n}_1,J_1),
\end{equation}
where $J_1 \coloneqq J \rvert_{\mathfrak{n}_1}$ and $\mathfrak{gl}(\mathfrak{n}_1,J_1)$ denotes endomorphisms of $\mathfrak{n}_1$ commuting with $J_1$. We denote $\mathfrak{k}\coloneqq \mathfrak{n}^{\perp_g}=\R e_{2n}$ and we say that the basis $\{e_1,\ldots,e_{2n}\}$ is \emph{adapted} to the splitting $\mathfrak{g} = J\mathfrak{k} \oplus \mathfrak{n}_1 \oplus \mathfrak{k}$.
The algebraic data $(a,v,A)$ fully characterizes the Hermitian structure $(J,g)$ and we we shall often denote the resulting Hermitian almost abelian Lie algebra by $(\mathfrak{g}(a,v,A),J,g)$.

Before studying the LCB condition, we recall the known characterizations for special Hermitian almost abelian Lie algebras.

\begin{proposition}
	\label{Herm-alm-ab}
A Hermitian almost abelian Lie algebra $(\mathfrak{g}(a,v,A),J,g)$ is
\begin{itemize}
	\item \makebox[1.35cm]{K\"ahler,\hfill} if $v=0$, $A \in \mathfrak{u}(\mathfrak{n}_1,J_1,g)$ (see \cite{LV}),
	\item {\makebox[1.35cm]{LCK,\hfill} if $v=0$, $A\in \R\text{\normalfont Id}_{\mathfrak{n}_1}\oplus \mathfrak{u}(\mathfrak{n}_1,J_1,g)$ or $n=2$, $A=0$ (see \cite{AO}),}
	\item {\makebox[1.35cm]{balanced,\hfill} if $v=0$, $\operatorname{tr} A=0$ (see \cite{FP1}),}
	\item {\makebox[1.35cm]{SKT,\hfill} if $[A,A^t]=0$ and the eigenvalues of $A$ have real part $-\frac{a}{2}$ or $0$ (see \cite{AL}),}
\end{itemize}
where $\mathfrak{u}(\mathfrak{n}_1,J_1,g)=\mathfrak{so}(\mathfrak{n}_1,g) \cap \mathfrak{gl}(\mathfrak{n}_1,J_1)$.
\end{proposition}

We also recall that, in terms of an adapted unitary basis, the Lee form of a Hermitian almost abelian Lie algebra $(\mathfrak{g}(a,v,A),J,g)$ is given by
\begin{equation} \label{leeform}
\theta=(Jv)^{\flat} - (\operatorname{tr} A)e^{2n},
\end{equation}
where the isomorphism $(\cdot)^\flat \colon \mathfrak{g} \to \mathfrak{g}^*$ is defined by $X^\flat\coloneqq g(X,\cdot)$, $X \in \mathfrak{g}$. See \cite{FP1} for details.

We are ready to prove the analogous characterization for LCB structures.
\begin{theorem} \label{LCB_avA}
A Hermitian almost abelian Lie algebra $(\mathfrak{g}(a,v,A),J,g)$ is LCB if and only if \mbox{$A^tv=0$}.
\end{theorem}
\begin{proof}
Observe that, given any $1$-form $\alpha \in \mathfrak{g}^*$, since
$d\alpha(X,e_{2n})=\alpha([e_{2n},X])$, $X \in \mathfrak{g}$, one has
\[
d\alpha=(\text{ad}_{e_{2n}}^* \alpha) \wedge e^{2n} = (a\alpha(e_1)+\alpha(v))\,e^1 \wedge e^{2n} + A^*(\alpha\rvert_{\mathfrak{n}_1}),
\]
with respect to the fixed adapted unitary basis $\{e_1,\ldots,e_{2n}\}$.
Then the exterior derivative of the Lee form \eqref{leeform} satisfies
\[
d\theta=g(v,Jv)\,e^1 \wedge e^{2n} + (A^tJv)^\flat \wedge e^{2n}=(A^tJv)^\flat \wedge e^{2n},
\]
where $A^t \in \mathfrak{gl}(\mathfrak{n}_1)$ is defined by $A^tX \coloneqq (A^*(X^\flat))^\sharp$, $X \in \mathfrak{n}_1$, $(\cdot)^\sharp$ denoting the inverse of $(\cdot)^\flat$. Then $d\theta$ vanishes if and only if $A^tJv=0$. $J_1$ commutes with $A$ and we have $J_1^t=-J_1$, so $J_1$ commutes with $A^t$ as well. The previous condition then reads $JA^tv=0$, which is equivalent to $A^tv=0$. 
\end{proof}

We note that the condition $A^tv=0$ is equivalent to $g(v,AX)=0$ for all $X \in \mathfrak{n}_1$. In particular, when $v \neq 0$, it implies $v \notin \operatorname{im} A$, so that $\text{rank}(v\lvert A)=\text{rank}(A)+1$, where $v \lvert A$ denotes the matrix obtained by juxtaposing $v$ and $A$.

In \cite{AL}, the authors determined a formula for the Bismut-Ricci form of a Hermitian almost abelian Lie algebra $(\mathfrak{g}(a,v,A),J,g)$, obtaining
\begin{equation}\label{rhoB}
\rho^B=-\left(a^2-\tfrac{1}{2}a\operatorname{tr}A+\lVert v \rVert^2\right)\, e^1 \wedge e^{2n} - (A^tv)^{\flat} \wedge e^{2n},
\end{equation}
in terms of the fixed adapted unitary basis $\{e_1,\ldots,e_{2n}\}$ (cf. also \cite{FP}).

The next result is a straightforward consequence of Theorem \ref{LCB_avA} and formula \eqref{rhoB}.
\begin{proposition}
A Hermitian almost abelian Lie algebra $(\mathfrak{g}(a,v,A),J,g)$ is LCB if and only if $\rho^B$ is of type $(1,1)$ (namely, $J\rho^B=\rho^B$), or equivalently if $\rho^B(X,Y)=0$ for every $X \in \mathfrak{n}_1$, $Y \in \mathfrak{g}$.
\end{proposition}

\section{Classification in dimension six}

We now focus on the six-dimensional case, with the goal of classifying almost abelian Lie algebras admitting LCB structures.
As recalled in the introduction, LCB structures generalize K\"ahler, balanced and LCK structures. Six-dimensional almost abelian Lie algebras carrying K\"ahler structures and balanced structures were classified in \cite{FP} and \cite{FP1} respectively. Therefore, before considering strictly LCB structures, we focus on the LCK condition.

In the following, we denote a Lie algebra via its structure equations: for example, the notation
\[
\mathfrak{g}_{4}=(f^{16},f^{26},f^{36},f^{46},0,0)
\] 
means that the Lie algebra $\mathfrak{g}_{4}$ is determined by a fixed basis $\{f_1,\ldots,f_6\}$ whose dual coframe $\{f^1,\ldots,f^6\}$ satisfies $df^1=f^{16}$, $df^2=f^{26}$, $df^3=f^{36}$, $df^4=f^{46}$, $df^5=df^6=0$, where $f^{ij}$ is a shorthand for the wedge product $f^i \wedge f^j$.

In \cite{Saw}, it was proven that a nilpotent Lie algebra admits an LCK structure if and only if it is isomorphic to $\mathfrak{h}_{2n+1} \oplus \R$, for some $n \geq 1$, where
\[
\mathfrak{h}_{2n+1}=\left(0,\ldots,0,\sum_{i=1}^{n} f^{2i-1} \wedge f^{2i}\right)
\]
denotes the $2n+1$-dimensional real Heisenberg algebra. In particular, the four-dimensional $\mathfrak{h}_3 \oplus \R = (0,0,0,f^{12})$ is the only one which is also almost abelian and, by \cite[Remark 3.4 (ii)]{AO} and \cite[Remarks 2.1, 2.3]{AngO}, one of the only two almost abelian Lie algebras admitting non-K\"ahler Vaisman metrics, up to isomorphism, the other one being $\mathfrak{aff}_2 \oplus 2\R$, where $\mathfrak{aff}_2 =(0,f^{12})$ denotes the two-dimensional real affine Lie algebra. In fact, every Hermitian metric on $\mathfrak{h}_3 \oplus \R$ and $\mathfrak{aff}_2 \oplus 2\R$ is Vaisman.

\begin{theorem} \label{LCK_class}
Let $\mathfrak{g}$ be a six-dimensional almost abelian Lie algebra. Then $\mathfrak{g}$ admits an LCK structure $(J,g)$, but no K\"ahler structures, if and only if it is isomorphic to one of the following:
\begin{itemize}
	\setlength{\itemindent}{-1em}
	\item[] $\mathfrak{g}_{1}\cm{=\mathfrak{k}_1^{p,p}}=(f^{16},pf^{26},pf^{36},pf^{46},pf^{56},0)$,  \,  $p \neq 0$,\smallskip
	\item[] $\mathfrak{g}_{2}\cm{=\mathfrak{k}_{8}^{p,q,q}}=(pf^{16},qf^{26},qf^{36},qf^{46}+f^{56},-f^{46}+qf^{56},0)$,  \,  $pq \neq 0$, \smallskip
	\item[] $\mathfrak{g}_{3}\cm{=\mathfrak{k}_{11}^{p,q,q,s}}=(pf^{16},qf^{26}+f^{36},-f^{26}+qf^{36},qf^{46}+rf^{56},-rf^{46}+qf^{56},0)$,  \,  $pq \neq 0$,  $r \neq 0$,\smallskip
	\item[] $\mathfrak{g}_{4}\cm{=\mathfrak{k}_{20}^1}=(f^{16},f^{26},f^{36},f^{46},0,0)$,\smallskip
	\item[] $\mathfrak{g}_{5}\cm{=\mathfrak{k}_{22}^{1,r}}=(f^{16},f^{26},f^{36}+rf^{46},-rf^{36}+f^{46},0,0)$, \, $r \neq 0$, \smallskip
	\item[] $\mathfrak{g}_{6}\cm{=\mathfrak{k}_{25}^{p,p,r}}=(pf^{16}+f^{26},-f^{16}+pf^{26},pf^{36}+rf^{46},-rf^{36}+pf^{46},0,0)$, \, $pr \neq 0$.
\end{itemize}
Among these, only the indecomposable Lie algebras $\mathfrak{g}_{1}^{p=-\frac{1}{4}}$\cm{$\mathfrak{k}_1^{-\frac{1}{4},-\frac{1}{4}}$}, $\mathfrak{g}_{2}^{p=-4q}$ \cm{$\mathfrak{k}_{8}^{p,-\frac{1}{4}p,-\frac{1}{4}p}$} and $\mathfrak{g}_{3}^{p=-4q}$ \cm{$\mathfrak{k}_{11}^{p,-\frac{1}{4}p,-\frac{1}{4}p,s}$} are unimodular. None of the corresponding Lie groups admit compact quotients by lattices, by \cite[Theorem 3.7]{AO}. \end{theorem} 
\begin{proof}
Let $(J,g)$ be an LCK structure on $\mathfrak{g}$. Let $\{e_1,\ldots,e_6\}$ be a unitary basis of $(\mathfrak{g},J,g)$ adapted to the splitting $\mathfrak{g}=J\mathfrak{k} \oplus \mathfrak{n}_1 \oplus \mathfrak{k}$, so that, by \cite{AO}, the matrix $B$ associated with $\text{ad}_{e_6}\rvert_{\mathfrak{n}}$ is of the form \eqref{B}, with
\begin{equation} \label{LCK}
v=0, \quad A= \lambda\,\text{Id}_{\mathfrak{n}_1} + U,\; \lambda \in \R,\,U \in \mathfrak{u}(\mathfrak{n}_1,J_1,g).
\end{equation}
Since $U$ is traceless, one must have $\lambda=\frac{\operatorname{tr} A}{4}$.
Following \cite[Theorem 3.2]{FP}, up to taking a different basis $\{e_2,\ldots,e_5\}$ for $\mathfrak{n}_1$ and rescaling $e_6$, the fact that $A$ commutes with $J_1$ forces $A$ to be represented by a real $4 \times 4$ matrix of one of the following types:
\begin{equation} \label{A}
A_1=\left( \begin{smallmatrix} p&0&0&0 \\ 0&p&0&0 \\ 0&0&q&0 \\ 0&0&0&q \end{smallmatrix} \right)\!\!, \,
A_2=\left( \begin{smallmatrix} p&1&0&0 \\ -1&p&0&0 \\ 0&0&q&0 \\ 0&0&0&q \end{smallmatrix} \right)\!\!, \,
A_3=\left( \begin{smallmatrix} p&1&0&0 \\ -1&p&0&0 \\ 0&0&q&r \\ 0&0&-r&q \end{smallmatrix} \right)\!\!, \,
A_4=\left( \begin{smallmatrix} p&1&0&0 \\ 0&p&0&0 \\ 0&0&p&1 \\ 0&0&0&p \end{smallmatrix} \right)\!\!, \,
A_5=\left( \begin{smallmatrix} p&1&-1&0 \\ -1&p&0&-1 \\ 0&0&p&1 \\ 0&0&-1&p \end{smallmatrix} \right)\!\!,
\end{equation}
$p,q,r \in \R$, with $r \neq 0$ to avoid redundancy.
All we need to do is determine which matrices $A_i$ in \eqref{A} can be decomposed as $\lambda \operatorname{Id} + U$ for some $\lambda \in \R$, $U \in \mathfrak{u}(\mathfrak{n}_1,J_1,g)$. For each $i=1,2,3,4,5$, consider the matrix $U_i=A_i - \frac{\operatorname{tr} A}{4}\operatorname{Id}$: for $i=4,5$, $U_i$ is never complex-diagonalizable (namely, diagonalizable as a complex matrix), so it cannot be skew-symmetric with respect to any metric; for $i=1,2,3$, the requirement that all the eigenvalues of $U_i$ should be pure imaginary imposes $p=q$, so that one is left with
\[
U_1=\left( \begin{smallmatrix} 0&0&0&0 \\ 0&0&0&0 \\ 0&0&0&0 \\ 0&0&0&0 \end{smallmatrix} \right)\!, \quad
U_2=\left( \begin{smallmatrix} 0&1&0&0 \\ -1&0&0&0 \\ 0&0&0&0 \\ 0&0&0&0 \end{smallmatrix} \right)\!, \quad
U_3=\left( \begin{smallmatrix} 0&1&0&0 \\ -1&0&0&0 \\ 0&0&0&r \\ 0&0&-r&0 \end{smallmatrix} \right)\!,
\]
all of which are skew-symmetric with respect to the standard metric and commute with
\[
J_1=\left( \begin{smallmatrix} 0&-1&0&0\\1&0&0&0 \\ 0&0&0&-1 \\ 0&0&1&0  \end{smallmatrix} \right).
\]
Completing the corresponding $A_i$ to the full matrix
\[
B=\begin{pmatrix} a & 0 \\ 0 & A_i \end{pmatrix}
\]
representing $\text{ad}_{e_6}\rvert_{\mathfrak{n}}$ and assuming $\operatorname{tr} A_i \neq 0$ to discard the K\"ahler cases,
one can easily see which algebras can be obtained:
\begin{itemize}
\item[] $A_1$ yields $\mathfrak{g}_{1}$ and $\mathfrak{g}_{4}$,
\item[] $A_2$ yields $\mathfrak{g}_{2}$ and $\mathfrak{g}_{5}$,
\item[] $A_3$ yields $\mathfrak{g}_{3}$ and $\mathfrak{g}_{6}$. \qedhere 
\end{itemize}
\end{proof}

\begin{theorem} \label{LCB_class}
Let $\mathfrak{g}$ be a six-dimensional almost abelian Lie algebra which does not admit balanced or LCK structures. If $\mathfrak{g}$ is nilpotent, then it admits an LCB structure $(J,g)$ if and only if it is isomorphic to one of the following:
\begin{itemize}
\setlength{\itemindent}{-1em}
\item[] $(0,0,0,0,0,f^{12})$,
\item[] $(0,0,0,f^{12},f^{13},f^{14})$.
\end{itemize}
If $\mathfrak{g}$ is non-nilpotent, then it admits an LCB structure $(J,g)$ if and only if it is isomorphic to one of the following:
\begin{itemize}
\setlength{\itemindent}{-1em}
\item[] $\lie_{1} \hspace{1.35mm} =(f^{16},pf^{26},pf^{36},qf^{46},qf^{56},0)$,  \,  $pr \neq 0$, $p \neq \pm q$, \smallskip
\item[] $\lie_{2} \hspace{1.35mm} =(f^{16},pf^{26}+f^{36},pf^{36},pf^{46}+f^{56},pf^{56},0)$, \, $p \neq 0$, \smallskip
\item[] $\lie_{3} \hspace{1.35mm} =(pf^{16},qf^{26},qf^{36},rf^{46}+f^{56},-f^{46}+rf^{56},0)$,  \,  $pq \neq 0$, $q \neq \pm r$, \smallskip
\item[] $\lie_{4} \hspace{1.35mm} =(pf^{16},qf^{26}+f^{36},-f^{26}+qf^{36},rf^{46}+sf^{56},-sf^{46}+rf^{56},0)$, \, $pqs \neq 0$, $q\neq \pm r$,\smallskip
\item[] $\lie_{5} \hspace{1.35mm} =(pf^{16},qf^{26}+f^{36}-f^{46},-f^{26}+qf^{36}-f^{56},qf^{46}+f^{56},-f^{46}+qf^{56},0)$, \,  $pq \neq 0$, \smallskip
\item[] $\lie_{6} \hspace{1.35mm} =(f^{16},f^{26},0,0,0,0)$, \smallskip
\item[] $\lie_{7} \hspace{1.35mm} = (f^{16},f^{26}+f^{36},f^{36},0,0,0)$, \smallskip
\item[] $\lie_{8} \hspace{1.35mm} =(pf^{16}+f^{26},-f^{16}+pf^{26},0,0,0,0)$, \, $p \neq 0$, \smallskip
\item[] $\lie_{9} \hspace{1.35mm} =(f^{16},pf^{26},pf^{36},0,0,0)$, \,  $p \neq 0$,\smallskip
\item[] $\lie_{10} =(pf^{16},qf^{26}+f^{36},-f^{26}+qf^{36},0,0,0)$,  \, $pq \neq 0$, \smallskip
\item[] $\lie_{11}                =(f^{16},f^{26},pf^{36},pf^{46},0,0)$, \,  $p \neq 0,\pm1$,\smallskip
\item[] $\lie_{12}                =(f^{16},f^{26},f^{46},0,0,0)$,\smallskip
\item[] $\lie_{13}                =(f^{16},f^{26},qf^{36}+rf^{46},-rf^{36}+qf^{46},0,0)$, \, $q \neq \pm 1$, $r \neq 0$, \smallskip
\item[] $\lie_{14}                =(pf^{16}+f^{26},-f^{16}+pf^{26},f^{46},0,0,0)$, \smallskip
\item[] $\lie_{15}                =(f^{16}+f^{26},f^{26},f^{36}+f^{46},f^{46},0,0)$, \smallskip
\item[] $\lie_{16}                =(pf^{16}+f^{26},-f^{16}+pf^{26},qf^{36}+rf^{46},-rf^{36}+qf^{46},0,0)$, \, $r \neq 0$, $p^2+q^2 \neq 0$, $p \neq \pm q$, \smallskip
\item[] $\lie_{17}                =(pf^{16}+f^{26}-f^{36},-f^{16}+pf^{26}-f^{46},pf^{36}+f^{46},-f^{36}+pf^{46},0,0)$, \, $p \neq 0$.
\end{itemize}
Among these, only $\lie_1^{q=-\frac{1}{2}-p}$,  $\lie_{2}^{p=-\frac{1}{4}}$, $\lie_{3}^{r=-\frac{p}{2}-q}$, $\lie_{4}^{r=-\frac{p}{2}-q}$, $\lie_{5}^{q=-\frac{p}{4}}$, $\lie_{9}^{p=-\frac{1}{2}}$, $\lie_{10}^{q=-\frac{p}{2}}$ and $\lie_{14}^{p=0}$ are unimodular.  \end{theorem}
\begin{proof}
Let $(J,g)$ be an LCB structure on $\mathfrak{g}$. As in Theorem \ref{LCK_class}, we need to examine each matrix $A_i$ in \eqref{A} to see whether they can satisfy the LCB condition $A_i^tv=0$ for some suitable metric and vector $v$. Of course $v=0$ is a sufficient condition and, in this case, after discarding the algebras admitting balanced or LCK structures (including the nilpotent $(0,0,0,0,f^{12},f^{13})$, which admits balanced structures, by \cite{Uga}), we have that
\begin{itemize}
\item[] $A_1$ yields $\lie_{1}$, $\lie_{6}$, $\lie_{9}$ and $\lie_{11}$,
\item[] $A_2$ yields $\lie_{3}$, $\lie_{8}$, $\lie_{10}$ and $\lie_{13}$, \item[] $A_3$ yields $\lie_{4}$ and $\lie_{16}$,
\item[] $A_4$ yields $\lie_{2}$ and $\lie_{15}$,
\item[] $A_5$ yields $\lie_{5}$ and $\lie_{17}$.
\end{itemize}
To complete the classification, we now assume $v \neq 0$. Then $A^tv=0$ forces $A$ to be degenerate: we are then left with $A_1^{q=0}$, $A_2^{q=0}$ (both with $p$ possibly vanishing) and $A_4^{p=0}$. If $a=g([e_6,e_1],e_1)$ is not an eigenvalue of $A_i$, then $\operatorname{im}(A-a\,\text{Id}_{\mathfrak{n}_1})=\mathfrak{n}_1$, so that $v=AX-aX$ for some $X \in \mathfrak{n}_1$ and the matrix $B$ corresponding to $\text{ad}_{e_6}\rvert_{\mathfrak{n}}$ can be brought into the form
\[
B=\begin{pmatrix} a & 0 \\ 0 & A_i \end{pmatrix},
\]
simply by replacing $e_1$ with $e_1^\prime=e_1-X$, so that eventually we get some of the previously found Lie algebras. Otherwise, if $a$ is an eigenvalue of $A_i$, the algebraic multiplicity of $a$ as an eigenvalue of $B$ (namely, its multiplicity as a root of the characteristic polynomial) might exceed its algebraic multiplicity for $A_i$ by one: this happens exactly when $v \notin \operatorname{im}(A-a\,\text{Id}_{\mathfrak{n}_1})$ and, in this case, $B$ is similar to a $5 \times 5$ matrix obtained by taking $A_i$ and raising the rank of a Jordan block relative to the eigenvalue $a$ by one: this can occur for $A_1^{q=0}$, when $a=p$ or $a=0$, for $A_2^{q=0}$ when $a=0$ and for $A_4^{p=0}$, $a=0$.

For the cases $A_{i}^{q=0}$, $i=1,2$, with $a=0$, one can simply assume that the basis $\{e_2,\ldots,e_5\}$ with respect to which $A_i$ is in the form \eqref{A} is orthonormal, with $Je_2=e_3$, $Je_4=e_5$, and take $v=e_4$, for instance. 

For $A_4^{p=0}$, assume again that $\{e_2,e_3,e_4,e_5\}$ is orthonormal, this time satisfying $Je_2=e_4$, $Je_3=e_5$ and take $v=e_3$, for example.

For the remaining case $A_1^{q=0}$, $a=p \neq 0$, one can consider for example the Hermitian almost abelian Lie algebra $(\mathfrak{g}(a,v,A),J,g)$ determined by the data
\[
a=p,\quad v= \left( \begin{smallmatrix} 0\\0\\1\\0 \end{smallmatrix} \right), \quad A=\left( \begin{smallmatrix} p&1&0&0\\0&0&0&0 \\ 0&0&0&0 \\ 0&0&1&p \end{smallmatrix} \right),\,p \neq 0,
\]
with respect to an adapted unitary basis $\{e_1,\ldots,e_6\}$, $Je_i=e_{7-i}$, $i=1,2,3$. Then, it is easy to check that $A$ is similar to $A_1^{q=0}$ and that $A^tv=0$, $v \notin \operatorname{im}(A-p\,\text{Id}_{\mathfrak{n}_1})$, so that the structure is LCB and the whole matrix $B$ is similar to
\[
\left(\begin{smallmatrix} p&1&0&0&0 \\ 0&p&0&0&0 \\ 0&0&p&0&0 \\ 0&0&0&0&0 \\ 0&0&0&0&0 \end{smallmatrix} \right)
\]
as desired.

These new cases with $v \neq 0$ yield Lie algebras isomorphic to $\lie_{7}$, $\lie_{12}$, $\lie_{14}$ or one of the two nilpotent Lie algebras of the statement, concluding the proof.
\end{proof}

\begin{remark}
It can be shown that, among the unimodular Lie groups whose Lie algebra appears in Theorem \ref{LCB_class}, the ones with Lie algebra
$\lie_1^{q=-\frac{1}{2}-p}$, $\lie_{2}^{p=-\frac{1}{4}}$, $\lie_{3}^{r=-\frac{p}{2}-q}$ and $\lie_{9}^{p=-\frac{1}{2}}$
do not admit any compact quotients by lattices.

We prove this only for $\lie_1^{q=-\frac{1}{2}-p}$, since the discussion for the other two Lie algebras is analogous.
Following \cite{Bock}, a co-compact lattice exists on such Lie groups if and only if there exists a non-zero $t_0 \in \R$ and a basis of $\mathfrak{n}$ such that the matrix associated with $\text{exp}(t_0\text{ad}_{f_6})\rvert_{\mathfrak{n}}$ has integer entries. In the basis $\{f_1,\ldots,f_5\}$ one easily computes 
\begin{equation} \label{exptad}
\text{exp}(t\,\text{ad}_{f_6})\rvert_{\mathfrak{n}}=\text{diag}\big(e^t,e^{pt},e^{pt},e^{-pt-\frac{1}{2}t},e^{-pt-\frac{1}{2}t}\big).
\end{equation}
Its minimal polynomial, namely the monic polynomial $P_t$ of least degree such that $P_t(\text{exp}(t\,\text{ad}_{f_6})\rvert_{\mathfrak{n}})=0$, is of the form $P_t(x)=\sum_{i=0}^3 a_i(t,p) x^i$, with coefficients
\[
a_0=-e^{\frac{t}{2}},\quad a_1=e^{t(1+p)}+e^{-\frac{t}{2}}+e^{t\left(\frac{1}{2}-p\right)},\quad a_2=-e^{pt}-e^t-e^{-t\left(\frac{1}{2}+p\right)},\quad a_3=1.
\]
If \eqref{exptad} is conjugate to an integer matrix for some $t_0$, then necessarily $P_{t_0}(x)$ is an integer polynomial, so that $a_0(t_0,p) \in \mathbb{Z}$ forces $t_0=2\log k$, for some $k \in \mathbb{Z}_{>0}$. 
Assuming $a_2(t_0,p) \in \mathbb{Z}$, one computes
\[
k^2\left(k^2+a_2(t_0,p)\right)+a_1(t_0,p)=\tfrac{1}{k},
\]
which is integer if and only if $k=1$, that is, $t_0=0$, a contradiction.

Instead, for some choices of the parameters, the Lie groups with Lie algebra $\lie_{4}^{r=-\frac{p}{2}-q}$, $\lie_{10}^{q=-\frac{p}{2}}$ and $\lie_{14}^{p=0}$ admit co-compact lattices (see \cite{FP} and the references therein). Some results are known for the remaining Lie groups, namely the ones corresponding to $\lie_{5}^{q=-\frac{p}{4}}$ \cm{$\mathfrak{k}_{12}^{p,-\frac{p}{4}}$} (see \cite{CM}), but the existence of lattices on them is still an open problem.
\end{remark}

\section{Compatibility results between Hermitian metrics} \label{sec_comp}

In this section, we ask whether a (unimodular) almost abelian Lie algebra endowed with a fixed complex structure may admit two different kinds of special Hermitian metrics.

In order to carry over the results to almost abelian solvmanifolds, we exploit the well-known ``symmetrization'' process. We summarize the results we need in the next lemma.
Recall that a solvable Lie group is called \emph{completely solvable} if all the eigenvalues of $\text{ad}_X$ are real, for every $X$ in its Lie algebra.
\begin{lemma} {\normalfont (\cite{Bel, FG, Uga, Saw, AU})} \label{symm}
Let $\Gamma \backslash G$ be a compact solvmanifold endowed with a left-invariant complex structure $(J,g)$. Then, the existence of a balanced (resp.\ SKT) metric implies the existence of a left-invariant balanced (resp.\ SKT) metric.
If $G$ is completely solvable, the analogous results hold for LCK and LCB metrics.
\end{lemma}

\subsection{SKT and LCB} The SKT condition and the balanced condition are two ``transversal'' generalizations of the K\"ahler condition. Indeed, by \cite{AI} a Hermitian metric which is both SKT and balanced is K\"ahler and it has been conjectured in \cite{FV} that a compact complex manifold admitting an SKT metric and a balanced metric necessarily admits a K\"ahler metric as well. For almost abelian solvmanifolds, the conjecture was proven in \cite{FP1}. 

The same transversality no longer holds when considering the weaker LCB condition instead of the balanced condition, and the same Hermitian metric can even be SKT and LCB at the same time: in \cite{FT1}, it was proven that
every non-K\"ahler compact homogeneous complex surface admits a compact torus bundle carrying an SKT and LCB metric; moreover, an example of compact nilmanifold in any even dimension admitting a left-invariant metric which is both SKT and LCB with respect to a fixed left-invariant complex structure was exhibited in \cite{OOS}.

In addition, recalling that LCK metrics are particular instances of LCB metrics, it is easy to see that a non-K\"ahler LCK almost abelian Lie algebra $(\mathfrak{g}(a,v,A),J,g)$ is also SKT if and only if it satisfies \eqref{LCK} with $a \neq 0$ and $\lambda=-\frac{a}{2}$ or if $n=2$, $\mathfrak{g} \cong \mathfrak{h}_3 \oplus \R$ or $\mathfrak{g} \cong \mathfrak{aff}_2 \oplus 2\R$ (cf. also \cite{FP1}).

\begin{proposition}
Let $\mathfrak{g}$ be an almost abelian Lie algebra endowed with a complex structure $J$. If $(\mathfrak{g},J)$ admits an SKT metric, then it admits an LCB metric as well.
\end{proposition}  
\begin{proof} Let $g$ denote the SKT metric. By \cite{AL}, with respect a unitary basis $\{e_1,\ldots,e_{2n}\}$ of $\mathfrak{g}$ adapted to the splitting $\mathfrak{g}=J\mathfrak{k} \oplus \mathfrak{n}_1 \oplus \mathfrak{k}$, the matrix $B$ associated with $\text{ad}_{e_{2n}}\rvert_{\mathfrak{n}}$ is of the form \eqref{B}, with $[A,A^t]=0$ and the eigenvalues of $A$ having real part equal to $-\frac{a}{2}$ or $0$.

Decompose $v \in \mathfrak{n}_1=\operatorname{im}(A-a\,\text{Id}_{\mathfrak{n}_1}) \oplus \left(\operatorname{im}(A-a\,\text{Id}_{\mathfrak{n}_1})\right)^{\perp_g}$ as $v=AX-aX+v^\prime$ for some $X \in \mathfrak{n}_1$ and $v^\prime \in \left(\operatorname{im}(A-a\,\text{Id}_{\mathfrak{n}_1})\right)^{\perp_g}$, that is, $(A-a\,\text{Id}_{\mathfrak{n}_1})^tv^\prime=0$.

Consider the new $J$-Hermitian metric $g^\prime=g\rvert_{\mathfrak{n}_1} + ({e^1}^\prime)^2 + ({e^{2n}}^\prime)^2$, with $e_1^\prime=e_1-X$, $e_{2n}^\prime=Je_1^\prime$. Then, the matrix $B^\prime$ associated with $\text{ad}_{e_{2n}^\prime}\rvert_{\mathfrak{n}}$ with respect to the new adapted unitary basis $\{e_1^\prime,e_2,\ldots,e_{2n-1}\}$ for $\mathfrak{n}$ is of the form
\[
B^\prime=\begin{pmatrix} a & 0 \\ v^\prime & A \end{pmatrix},
\]
with $A$ as above and $(A-a\,\text{Id}_{\mathfrak{n}_1})^tv^\prime=0$. If $a \neq 0$, $a$ is not an eigenvalue of $A$, so that $v^\prime=0$. Instead, if $a=0$, we have $A^tv^\prime=0$. In either case, the metric $g^\prime$ is LCB.
\end{proof}

Using Lemma \ref{symm}, we get
\begin{corollary}
Let $\Gamma \backslash G$ be a compact almost abelian solvmanifold endowed with a left-invariant complex structure $J$. If $(\Gamma \backslash G,J)$ admits an SKT metric, then it admits an LCB metric as well.
\end{corollary}

\begin{example} \label{exampleSKTLCB}
We now exhibit an example of compact almost abelian solvmanifold in any even dimension admitting a (left-invariant) Hermitian structure which is at the same time SKT and LCB. For any $n \geq 2$, consider the $2n$-dimensional simply connected unimodular almost abelian Lie group $S_{2n}$ having indecomposable Lie algebra $\mathfrak{s}_{2n}$ endowed with a fixed coframe $\{e^1,\ldots,e^{2n}\}$ satisfying the structure equations
\begin{gather*}
de^1=a\,e^1 \wedge e^{2n},\quad de^2=-\tfrac{a}{2} \, e^2 \wedge e^{2n} + e^3 \wedge e^{2n},\quad de^3=- e^2 \wedge e^{2n} - \tfrac{a}{2} \, e^3 \wedge e^{2n}, \quad de^{2n}=0, \\
de^{2i}=c \, e^{2i+1} \wedge e^{2n},\quad de^{2i+1}=-c \, e^{2i} \wedge e^{2n},\quad i=2,\ldots,n-1,
\end{gather*}
for some $a,c \in \R - \{0\}$, with $c$ depending on $a$ in a way which we shall explain. 
Now, it is easy to check that the left-invariant Hermitian structure $(J,g)$ on $S_{2n}$ defined by
\[
Je_1=e_{2n}, \quad Je_{2i}=e_{2i+1},\,i=1,\ldots,n-1, \quad g=\sum_{i=1}^{2n} (e^i)^2,
\]
is both SKT and LCB, satisfying in particular $v=0$.
	
As we shall now show, $S_{2n}$ admits compact quotients by lattices, for all $n$, for some values of $a$ and $c$: by \cite{Bock}, this is equivalent to proving that there exists $t_0 \in \R-\{0\}$ such that $\text{exp}(t_0B_{2n})$ is similar to an integer matrix, $B_{2n}$ being the $(2n-1) \times (2n-1)$ matrix representing $\text{ad}_{e_{2n}}\rvert_{\mathfrak{n}}$ in the fixed basis. The claim is true for $n=2$ for countably many values of $a \in \R-\{0\}$, with compact quotients of $S_4$ biholomorphic to Inoue surfaces (see \cite[Section 3.2.2]{AO} for a detailed discussion). Fixing $a \in \R$ such that $S_4$ admits co-compact lattices, let $t_0\in \R-\{0\}$ be such that $\text{exp}(t_0B_{4})$ is similar to an integer matrix and set $c \coloneqq \tfrac{2\pi}{t_0}$, so that
\[
\text{exp}\left(t_0 \left( \begin{smallmatrix} 0 & c \\ -c & 0 \end{smallmatrix} \right)\right)=\begin{pmatrix} 1 & 0 \\ 0 & 1 \end{pmatrix}
\]
is an integer matrix. The claim then easily follows in any dimension by induction.
\end{example}

\subsection{Balanced and LCK}
By \cite{AO}, almost abelian Lie groups which admit left-invariant LCK structures and compact quotients by lattices only exist in real dimension four. The resulting solvmanifolds are biholomorphic to primary Kodaira surfaces, Inoue surfaces, hyperelliptic surfaces or complex tori: out of these, the only ones admitting K\"ahler metrics (recall that K\"ahler is equivalent to balanced, in real dimension four) are
complex tori or hyperellyptic surfaces, which, by \cite{HK}, cannot admit non-K\"ahler LCK metrics. Thus, we phrase the next result only in terms of structures on Lie algebras and not on compact almost abelian solvmanifolds, where the situation is already completely understood.

It was proven in \cite{OOS} that a nilpotent Lie algebra cannot admit a balanced metric and a non-K\"ahler LCK metric both compatible with the same complex structure. In the almost abelian setting, the situation is analogous, apart from one exception in the non-unimodular case.
\begin{proposition} \label{prop_balLCK}
	Let $\mathfrak{g}$ be an almost abelian Lie algebra endowed with a complex structure $J$. If $(\mathfrak{g},J)$ admits a balanced metric, then it does not admit any non-K\"ahler LCK metrics, unless $\mathfrak{g} \cong \mathfrak{aff}_2 \oplus 2\R$.
\end{proposition}
\begin{proof}
	As we have recalled in Proposition \ref{Herm-alm-ab}, an LCK almost abelian Lie algebra $(\mathfrak{g}(a,v,A),J,g)$ can either satisfy
	\eqref{LCK} or $n=2$, $A=0$, which corresponds to $\mathfrak{g} \cong \mathfrak{h}_3 \oplus \R$ (if $a=0$, $v \neq 0$), $\mathfrak{g} \cong \mathfrak{aff}_2 \oplus 2\R$ (if $a \neq 0$) or to $4\R$ (if $a=0$, $v=0$).
	
	Let $g$ denote an LCK metric and assume \eqref{LCK}. The result readily follows by observing that, in order to admit a balanced metric, $(\mathfrak{g},J)$ must satisfy $\operatorname{tr} A=\operatorname{tr} \text{ad}_X \rvert_{\mathfrak{n}_1}=0$ for all $X \in \mathfrak{g}$, so that $A \in \mathfrak{u}(\mathfrak{n}_1,J_1,g)$. This implies that $g$ is K\"ahler.
	
	If $n=2$ (in which case balanced implies K\"ahler) and we have $A=0$, note that $\mathfrak{h}_3 \oplus \R$ does not admit K\"ahler structures, while all Hermitian structures on $4\R$ are K\"ahler. On $\mathfrak{aff}_2 \oplus 2\R = (f^{12},0,0,0)$, consider the complex structure defined by $Jf_1=f_2$, $Jf_3=f_4$. The Hermitian metric $g=\sum_{i=1}^4 (f^i)^2$ is K\"ahler, while, denoting $f^i \odot f^j=\frac{1}{2}(f^i \otimes f^j +f^j \otimes f^i)$,
	\[
	g^\prime=2(f^1)^2+2(f^2)^2+(f^3)^2+(f^4)^2+2\,f^1 \odot f^3+2\,f^2 \odot f^4
	\]
	is non-K\"ahler and LCK, with fundamental form $\omega^\prime=2f^{12}+e^{14}-e^{23}+e^{34}$ satisfying $d\omega^\prime=f^{124}=(f^2+f^4) \wedge \omega^\prime$, $d(f^2+f^4)=0$.
\end{proof}

\subsection{Balanced and LCB}
Balanced metrics are trivially LCB. One could ask whether there exist non-K\"ahler compact solvmanifolds endowed with a left-invariant complex structure admitting both balanced metrics and non-balanced LCB metrics. 

For nilmanifolds, the answer is affirmative, as shown in \cite{OOS}.
As a corollary of the next proposition, the analogous result is not true for completely solvable almost abelian solvmanifolds.
\begin{proposition} \label{prop_balLCB}
Let $\mathfrak{g}$ be a unimodular almost abelian Lie algebra endowed with a complex structure $J$. If $(\mathfrak{g},J)$ carries a balanced metric, then it cannot admit any non-balanced LCB metrics.
\end{proposition}
\begin{proof}
By the characterization of unimodular balanced almost abelian Lie algebras, we know that $[\mathfrak{g},\mathfrak{g}] \subset \mathfrak{n}_1$, since there exists an adapted unitary basis with respect to the balanced metric satisfying $a=0$, $v=0$. In particular, $\text{rank}(\text{ad}_X)=\text{rank}(A)$ for all $X \in \mathfrak{g}$ transverse to $\mathfrak{n}$. Assume a non-balanced LCB metric $g$ exists. Then, any adapted unitary basis for $(\mathfrak{g},J,g)$ satisfies $a=0$, $\operatorname{tr} A=0$, $A^tv=0$, with $v \neq 0$ to ensure the metric is non-balanced. Now, this implies $\text{rank}(\text{ad}_X)=\text{rank}(A)+1$ for all $X \in \mathfrak{g}$ transverse to $\mathfrak{n}$, since $v \notin \operatorname{im} A$, a contradiction.
\end{proof}

Recalling Lemma \ref{symm}, we obtain
\begin{corollary}
Let $\Gamma \backslash G$ be a completely solvable almost abelian solvmanifold endowed with a left-invariant complex structure $J$. If $(\Gamma \backslash G,J)$ carries a balanced metric, then it cannot admit any non-balanced LCB metrics.
\end{corollary}

\begin{remark}
We note that Proposition \ref{prop_balLCB} is no longer true if one drops the hypothesis of unimodularity: this is clear from the example on the four-dimensional Lie algebra $\mathfrak{aff}_2 \oplus 2\R$ in the proof of Proposition \ref{prop_balLCK}, recalling that LCK implies LCB. However, one can easily find other examples of complex structures of higher-dimensional almost abelian Lie algebras admitting both balanced and non-balanced LCB metrics: consider the six-dimensional almost abelian Lie algebra (see \cite{FP1})
\[
\mathfrak{b}_2=(f^{16},f^{36},0,f^{56},0,0),
\]
endowed with the complex structure defined by $Jf_1=f_6$, $Jf_2=f_4$, $Jf_3=f_5$. On it, one has the balanced metric $g=\sum_{i=1}^6 (f^i)^2$ and the non-balanced and non-LCK LCB metric
\[
g^\prime=3(f^1)^2+(f^2)^2+(f^3)^2+(f^4)^2+(f^5)^2+3(f^6)^2+2(f^1 \odot f^2 + f^1 \odot f^3 + f^4 \odot f^6 + f^5 \odot f^6),
\]
whose associated Lee form is the closed $1$-form $\theta^\prime=f^5+f^6$, as shown by a direct computation.
\end{remark}

\section{Locally conformally hyperk\"ahler metrics} \label{LCHK}
We now turn our attention to the study of (locally conformally) hyperk\"ahler metrics on almost abelian Lie algebras.

In the nilpotent setting, these structures were studied in \cite{OOS}, where it was proven that compact nilmanifolds never admit left-invariant LCHK structures, unless they are tori.

In the next theorem, we classify almost abelian Lie algebras admitting LCHK structures. Recall that the spectrum of a matrix (or an endomorphism) $D$, denoted by $\operatorname{Spec}(D)$, is the set of its eigenvalues. Given $z \in \mathbb{C}$, we denote by $m_D(z)$ its algebraic multiplicity for $D$, namely its multiplicity as a root of the characteristic polynomial of $D$. When $D$ is complex-diagonalizable, $m_D(z)$ is also equal to the (complex) dimension of the corresponding eigenspace.

\begin{theorem}
A $4m$-dimensional almost abelian Lie algebra $\mathfrak{g}=\R^{4m-1} \rtimes_D \R$ admits an LCHK structure if and only if $D \in \mathfrak{gl}_{4m-1}$ is complex-diagonalizable and
\begin{itemize}
	\item[\normalfont (i)] $\operatorname{Spec}(D) \subset a + \R i$, for some $a \in \R$,
	\item[\normalfont (ii)] $m_D(a) \geq 3$,
	\item[\normalfont (iii)] $m_D(a+ib) \in 2 \mathbb{Z}$, for every $b \in \R-\{0\}$.
\end{itemize}
The Lie algebra $\mathfrak{g}$ admits a hyperk\"ahler structure if and only if the above holds, with $a=0$, in which case $\mathfrak{g}$ is unimodular and decomposable ($\mathfrak{g}=\mathfrak{g}^\prime \oplus (m_D(0))\R$, with $\mathfrak{g}^\prime$ indecomposable) and every LCHK structure on $\mathfrak{g}$ is hyperk\"ahler. In particular, there do not exist unimodular almost abelian Lie algebras admitting non-hyperk\"ahler LCHK structures.
\end{theorem}
\begin{proof}
Assume $\mathfrak{g}$ admits an LCHK structure $(I_1,I_2,I_3,g)$. In particular, $(I_1,g)$ is an LCK structure, so that there exists an adapted $(I_1,g)$-unitary basis $\{e_1,\ldots,e_{4m}\}$ of $\mathfrak{g}$ such that the matrix $B$ associated with $\text{ad}_{e_{4m}}\rvert_{\mathfrak{n}}$ is of the form \eqref{B} with the conditions \eqref{LCK} or $m=1$, $A=0$, by \cite{AO}. Note that, up to conjugation, we obtain the same $B$ when considering a basis adapted to $(I_2,g)$ or $(I_3,g)$.

Assume \eqref{LCK} holds. It follows that $B$ is complex-diagonalizable, hence $D$ is. Moreover, $\operatorname{Spec}(B) \subset \{a\} \cup \lambda + \R i$. We claim that $\lambda=a$: if this were not the case, in order for $(I_2,g)$ to be LCK, one should have that $I_1(e_{4m})=\pm I_2(e_{4m})$, implying $I_3(e_{4m})=I_1I_2(e_{4m})= \pm e_{4m}$, contradicting $I_3^2=-\text{Id}_{\mathfrak{g}}$. For the same reason, it follows that $m_{B}(a) \geq 3$, to accommodate for the fact that $I_1(e_{4m})$, $I_2(e_{4m})$ and $I_3(e_{4m})$ should be eigenvectors for $B$ with real eigenvalue, hence equal to $a$.

Now, denote by $V_z \subset \mathfrak{n} \otimes \mathbb{C}$ the eigenspace for $B$ corresponding to the eigenvalue $z \in \mathbb{C}$, and define 
\[
\mathfrak{m} \coloneqq \mathfrak{n} \cap I_1\mathfrak{n} \cap I_2 \mathfrak{n} \cap I_3 \mathfrak{n}=\text{span}\left<e_{4m},I_1(e_{4m}),I_2(e_{4m}),I_3(e_{4m})\right>^{\perp_g}.
\]
We note that $I_1$, $I_2$, $I_3$ must preserve $W_z \coloneqq (V_{z}+\overline{V_{z}}) \cap \mathfrak{m}$, for all $z \in \text{Spec}(B)$, since $\text{ad}_{e_{4m}}\rvert_{\mathfrak{m}}$ commutes with the restriction of each of the three complex structures on $\mathfrak{m}$. Note that, when $z$ is not real, we have that $W_z$ is the set of real elements of $V_{z}\oplus \overline{V_{z}}$. It follows that $W_z$ inherits a hyperhermitian structure, so that its real dimension is a multiple of four, which implies that the complex dimension of $V_z$ is a multiple of $2$, when $z$ is not real. Up to rescaling $B$ to recover $D$, points (i), (ii) and (iii) of the statement follow. 

Assume now that $B$ satisfies $A=0$, with $m=1$. In particular we have that $\mathfrak{g}$ is four-dimensional, with $\dim [\mathfrak{g},\mathfrak{g}]=1$, so, by \cite[Proposition 3.2]{Bar}, $\mathfrak{g}$ does not admit hypercomplex structures.

Conversely, assume $\mathfrak{g}=\R^{4m-1} \rtimes_D \R$ satisfies $D$ being complex-diagonalizable and requirements (i), (ii) and (iii). It follows that, up to a change of basis $\{e_1,\ldots,e_{4m-1}\}$ of $\R^{4m-1}$, $D$ is of the form
\begin{equation} \label{hyperD}
D=\text{diag}\left(C_{1},C_{2},\ldots,C_{m-1},a,a,a\right),
\end{equation}
where, for $i=1,\ldots,m-1$, $C_i$ is a $4 \times 4$ matrix of the form
\[
C_i= \left(\begin{smallmatrix} 
a & b_i & 0 & 0 \\
-b_i & a & 0 & 0 \\
0 & 0 & a & -b_i \\
0 & 0 & b_i & a
\end{smallmatrix}\right),
\]
for some (possibly vanishing) $b_i \in \R$. Denoting by $e_{4m}$ the generator of the extra $\R$, an explicit LCHK structure on $\mathfrak{g}$ is given by $(I_1,I_2,I_3,g)$, with $g=\sum_{i=1}^{4m} (e^i)^2$ and, with respect to the fixed basis, $I_i=\text{diag}(K_i,\ldots,K_i)$, $i=1,2,3$, with
\[
K_1=\left( \begin{smallmatrix}
	0&-1&0&0\\
	1&0&0&0\\
	0&0&0&-1 \\
	0&0&1&0
\end{smallmatrix} \right), \quad
K_2=\left( \begin{smallmatrix}
	0&0&-1&0\\
	0&0&0&1\\
	1&0&0&0 \\
	0&-1&0&0
\end{smallmatrix} \right), \quad
K_3=\left( \begin{smallmatrix}
	0&0&0&-1\\
	0&0&-1&0\\
	0&1&0&0 \\
	1&0&0&0
\end{smallmatrix} \right).
\]
The three induced Lee forms are all equal to the closed $1$-form $\theta=-(4m-2)ae^{4m}$.

The part of the claim regarding hyperk\"ahler structures easily follows from the fact that the K\"ahler condition on almost abelian Lie algebras corresponds to \eqref{LCK}, with $\lambda=0$. In particular, in this case, we note that, if $D$ is of the form \eqref{hyperD}, with $m_D(0)=3+4h$, we can assume $b_i=0$, $i=m-h,\ldots,m-1$, so that $\text{span}\left<e_{4(m-h)-3},\ldots,e_{4m-1}\right>$ is an abelian subalgebra of dimension $m_D(0)=3+4h$, while its complement, $\text{span}\left<e_1,\ldots,e_{4(m-h-1)},e_{4m}\right>$, is an indecomposable almost abelian Lie algebra.
\end{proof}

The previous theorem can be used to get a more precise list of almost abelian Lie algebras admitting hyperk\"ahler or LCHK structures: in the next proposition, we cover dimensions $4$, $8$ and $12$.

\begin{proposition}
Let $\mathfrak{g}$ be a $4m$-dimensional almost abelian Lie algebra.
\begin{itemize}[leftmargin=3em, itemindent=-1em]
\item
If $m=1$, $\mathfrak{g}$ admits a hyperk\"ahler structure if and only if $\mathfrak{g}=4\R$, while it admits a non-hyperk\"ahler LCHK structure if and only if it is isomorphic to $(f^{14},f^{24},f^{34},0)$. \smallskip
\item
If $m=2$, $\mathfrak{g}$ admits a hyperk\"ahler structure if and only if it is isomorphic to one among \smallskip
\begin{itemize}
	\setlength{\itemindent}{-1em}
	\item[] $8\R=(0,0,0,0,0,0,0,0)$,\smallskip
	\item[] $(f^{28},-f^{18},f^{48},-f^{38},0,0,0,0)$, \smallskip
\end{itemize}
while it admits a non-hyperk\"ahler LCHK structure if and only if it is isomorphic to \smallskip
\begin{itemize}
	\setlength{\itemindent}{-1em}
	\item[] $(f^{18},f^{28},f^{38},f^{48},f^{58},f^{68},f^{78},0)$,\smallskip
	\item[] $(f^{18},f^{28},f^{38},f^{48}+pf^{58},-pf^{48}+f^{58},f^{68}+pf^{78},-pf^{68}+f^{78},0)$, \,$p \neq 0$.
\end{itemize}\smallskip
\item
If $m=3$, $\mathfrak{g}$ admits a hyperk\"ahler structure if and only if it is isomorphic to one among \smallskip
\begin{itemize}
	\setlength{\itemindent}{-1em}
	\item[] $12\R=(0,0,0,0,0,0,0,0,0,0,0,0)$,\smallskip
	\item[] $(f^{2,12},-f^{1,12},f^{4,12},-f^{3,12},0,0,0,0,0,0,0,0)$,\smallskip
	\item[] $(f^{2,12},-f^{1,12},f^{4,12},-f^{3,12},pf^{6,12},-pf^{5,12},pf^{8,12},-pf^{7,12},0,0,0,0)$, \,$p \neq 0$, \smallskip
\end{itemize}
while it admits a non-hyperk\"ahler LCHK structure if and only if it is isomorphic to \smallskip
\begin{itemize}
\setlength{\itemindent}{-1em}
	\item[] $(f^{1,12},f^{2,12},f^{3,12},f^{4,12},f^{5,12},f^{6,12},f^{7,12},f^{8,12},f^{9,12},f^{10,12},f^{11,12},0)$,\smallskip
    \item[] $(f^{1,12},f^{2,12},f^{3,12},f^{4,12},f^{5,12},f^{6,12},f^{7,12},f^{8,12}+pf^{9,12},-pf^{8,12}+f^{9,12},f^{10,12}+pf^{11,12},\\-pf^{10,12}+f^{11,12},0)$, \,$p \neq 0$,\smallskip
    \item[] $(f^{1,12},f^{2,12},f^{3,12},f^{4,12}+pf^{5,12},-pf^{4,12}+f^{5,12},f^{6,12}+pf^{7,12},-pf^{6,12}+f^{7,12},f^{8,12}+qf^{9,12},\\-qf^{8,12}+f^{9,12},f^{10,12}+qf^{11,12},-qf^{10,12}+f^{11,12},0)$, \,$pq \neq 0$.
\end{itemize}
\end{itemize}
\end{proposition}

%\enlargethispage{1in}

\end{document}